\documentclass[reqno,a4paper]{amsart}
\usepackage{amssymb}
\usepackage{amsmath}
\usepackage{amsfonts}

\newtheorem{theorem}{Theorem}[section]
\theoremstyle{plain}

\newtheorem{corollary}{Corollary}[section]

\newtheorem{definition}{Definition}[section]

\newtheorem{lemma}{Lemma}[section]

\numberwithin{equation}{section}
\input{tcilatex}

\begin{document}
\title{A new kind of slant helix in Pseudo-Riemannian Manifolds}
\author{Evren Z\i plar}
\address{ Department of Mathematics, Faculty of Science, \c{C}ank\i r\i \
Karatekin University, \c{C}ank\i r\i , Turkey}
\email{evrenziplar@karatekin.edu.tr}
\urladdr{}
\author{Yusuf Yayl\i }
\address{Department of Mathematics, Faculty of Science, University of
Ankara, Tando\u{g}an, Turkey}
\email{yayli@science.ankara.edu.tr}
\author{\.{I}smail G\"{o}k}
\address{Department of Mathematics, Faculty of Science, University of
Ankara, Tando\u{g}an, Turkey}
\email{igok@science.ankara.edu.tr}
\urladdr{}
\date{May 30, 2013.}

\begin{abstract}
In this paper, we define a new kind of slant helix called $f$-eikonal $V_{n}$%
-slant helix in Pseudo- Riemannian manifolds and give the definition of
harmonic curvature functions related to the $f$-eikonal $V_{n}$-slant helix
in Pseudo- Riemannian manifolds. Moreover, we give some characterizations of 
$f$-eikonal $V_{n}$-slant helix by making use of the harmonic curvature
functions.
\end{abstract}

\maketitle

\section{\textbf{Introduction}}

Curves theory is an important framework in the differential geometry
studies. Helix is one of the most fascinating curves because we see helical
structure in nature, science and mechanical tools. Helices arise in the
field of computer aided design, computer graphics, the simulation of
kinematic motion or design of highways, the shape of DNA and carbon
nonotubes. Also, we can see the helical structure in fractal geometry, for
instance hyperhelices ( \cite{jain, scarr, yin}).

A curve of constant slope or general helix in Euclidean 3-space $E^{3},$ is
defined by the property that its tangent vector field makes a constant angle
with a fixed straight line (the axis of general helix). A classical result
stated by Lancret in 1802 and first proved by de Saint Venant in 1845 (\cite%
{Lancret} and \cite{Struik}) is: A necessary and sufficient condition that a
curve be a general helix is that the ratio of curvature to torsion be
constant. In \cite{haci}$,$ \"{O}zdamar and Hac\i saliho\u{g}lu defined
harmonic curvature functions $H_{i}$ $\left( 1\leq i\leq n-2\right) $ of a
curve $\alpha $ and generalized helices in $E^{3}$ to in $n-$dimensional
Euclidean space $E^{n}$. \ Moreover, they gave a characterization for the
inclined curves in $E^{n}$ : 
\begin{equation}
\text{\textquotedblleft A curve is an inclined curve if and only if }\dsum
\limits_{i=1}^{n}H_{i}^{2}=\text{constant\textquotedblright }
\end{equation}

Izumiya and Takeuchi defined a new kind of helix (slant helix) and they gave
a characterization of slant helices in Euclidean $3-$space $E^{3}$ \cite%
{izumiya}. In 2008, \"{O}nder \emph{et al}. defined a new kind of slant
helix in Euclidean $4-$space $E^{4}$ which is called $B_{2}-$slant helix and
they gave some characterizations of these slant helices in Euclidean $4-$%
space $E^{4}$ \cite{onder} . And then in 2009, G\"{o}k \emph{et al}.
generalized $B_{2}-$slant helix in $E^{4}$ to $E^{n}$, $n>3$, called $V_{n}-$%
slant helix in Euclidean and Minkowski $n$-space (\cite{gok, gok1}). Lots of
authors in their papers have investigated inclined curves and slant helices
using the harmonic curvature functions in Euclidean and Minkowski $n$-space (%
\cite{ali, ali1, klah, oz}). But, \c{S}enol et al.(\cite{zip}) see for the
first time that the characterization of inclined curves and slant helices in 
$(1.1)$ is true only for the case necessity but not true for the case
sufficiency in Euclian $n$-space. Then, they consider the
pre-characterizations about inclined curves and slant helices and
restructure them with the necessary and sufficient condition.

Let $M$ be a Riemannian manifold, where $\left \langle ,\right \rangle $ is
the metric. Let $f:M\rightarrow 
\mathbb{R}
$ be a function and let $\nabla f$ be its gradient, i.e., $%
df(X)=\left
\langle \nabla f,X\right \rangle $. We say that $f$ is eikonal
if it satisfies: $\left \Vert \nabla f\right \Vert $ is constant \cite{scala}
. $\nabla f$ is used many areas of science such as mathematical physics and
geometry. So, $\nabla f$ is very important subject. For example, the
Riemannian condition $\left \Vert \nabla f\right \Vert ^{2}=1$ (for
non-constant $f$ on connected $M$) is precisely the eikonal equation of
geometrical optics. Thus on a connected $M$, a non-constant real valued $f$
is Riemannian iff $f$ satisfies this eikonal equation. In the geometrical
optical interpretation, the level sets of $f$ are interpreted as wave
fronts. The characteristics of the eikonal equation (as a partial
differential equation), are then the solutions of the gradient flow equation
for $f$ (an ordinary differential equation), $x^{\prime }=\func{grad}f(x)$,
which are geodesics of $M$ orthogonal to the level sets of $f$, and which
are parametrized by arc length. These geodesics can be interpreted as light
rays orthogonal to the wave fronts \cite{Fischer} .

In this paper, we define $f$-eikonal $V_{n}$-slant helix in $n$-dimensional
pseudo-Riemannian manifolds and give the definition of harmonic curvature
functions related to $f$-eikonal $V_{n}$-slant helix in $n$-dimensional
pseudo-Riemannian manifolds. Moreover, we give some characterizations of $f$%
-eikonal $V_{n}$-slant helix by making use of the harmonic curvature
functions.

\section{\textbf{Preliminaries}}

In this section, we give some basic definitions from differential geometry.

\begin{definition}
A metric tensor $g$ on a smooth manifold $M$ is a symmetric non-degenerate
(0,2) tensor field on $M$.

In other words, $g\left( X,Y\right) =g\left( Y,X\right) $ for all $X,Y\in TM$
(tangent bundle) and at the each point $p$ of $M$ if $g\left(
X_{p},Y_{p}\right) =0$ for all $Y_{p}\in T_{p}\left( M\right) $ , then $%
X_{p}=0$ (non-degenerate), where $T_{p}\left( M\right) $ is the tangent
space of $M$ at the point $p$ and $g:T_{p}\left( M\right) \times T_{p}\left(
M\right) \rightarrow 
\mathbb{R}
$ \cite{neill} .
\end{definition}

\begin{definition}
A pseudo-Riemannian manifold (or semi-Riemannian manifold) is a smooth
manifold $M$ furnished with a metric tensor $g$. That is, a
pseudo-Riemannian manifold is an ordered pair $\left( M,g\right) $ \cite%
{neill} .
\end{definition}

\begin{definition}
We shall recall the notion of a proper curve of order $n$ in a $n$%
-dimensional pseudo-Riemannian manifold $M$ with the metric tensor $g$. Let $%
\alpha :I\rightarrow M$ be a non-null curve in $M$ parametrized by the
arclength $s$, where $I$ is an open interval of the real line $%
\mathbb{R}
$. We denote the tangent vector field of $\alpha $ by $V_{1}$. We assume
that $\alpha $ satisfies the following Frenet formula: 
\begin{eqnarray*}
\nabla _{V_{1}}V_{1} &=&k_{1}V_{2}, \\
\nabla _{V_{1}}V_{i} &=&-\varepsilon _{i-2}\varepsilon
_{i-1}k_{i-1}V_{i-1}+k_{i}V_{i+1},\text{ }1<i<n \\
\nabla _{V_{1}}V_{n} &=&-\varepsilon _{n-2}\varepsilon _{n-1}k_{n-1}V_{n-1},
\end{eqnarray*}%
where%
\begin{eqnarray*}
k_{1} &=&\left \Vert \nabla _{V_{1}}V_{1}\right \Vert >0 \\
k_{i} &=&\left \Vert \nabla _{V_{1}}V_{i}+\varepsilon _{i-2}\varepsilon
_{i-1}k_{i-1}V_{i-1}\right \Vert >0,\text{ \ }2\leq i\leq n-1 \\
\varepsilon _{j-1} &=&g\left( V_{j},V_{j}\right) \text{ }\left( =\pm
1\right) \text{ },\text{ }1\leq j\leq n,\text{on }I\text{, }
\end{eqnarray*}%
and $\nabla $ is Levi-Civita connection of $M$.

We call such a curve a proper curve of order $n$, $k_{i}$ $\left( 1\leq
i\leq n-1\right) $ its $i-th$ curvature and $V_{1},...,V_{n}$ its Frenet
Frame field.

Morever, let us recall that $\left \Vert X\right \Vert =\sqrt{\left \vert
g\left( X,X\right) \right \vert }$ for $X\in TM$, where $TM$ is the tangent
bundle of $M$ \cite{song} .
\end{definition}

\section{$f$\textbf{-eikonal }$V_{n}$\textbf{-slant helix curves in
pseudo-Riemannian manifolds}}

In this section, we define $f$-eikonal $V_{n}$-slant helix curves and we
give characterizations for a $f$-eikonal $V_{n}$-slant helix curve in $n$%
-dimensional pseudo-Riemannian manifold $M^{n}$ by using harmonic curvature
functions in terms of $V_{n}$ of the curve.

\begin{definition}
\textbf{\ }Let $M$ be a $n$-dimensional pseudo-Riemannian manifold and let $%
\alpha \left( s\right) $ be a proper curve of order $n$ (non-null) with the
curvatures $k_{i}$ $\left( i=1,...,n-1\right) $ in $M$. Then, harmonic
curvature functions of $\alpha $ are defined by 
\begin{equation*}
H_{i}^{\ast }:I\subset 
\mathbb{R}
\rightarrow 
\mathbb{R}%
\end{equation*}%
along $\alpha $ in $M$, where%
\begin{eqnarray*}
H_{0}^{\ast } &=&0, \\
H_{1}^{\ast } &=&\varepsilon _{n-3}\varepsilon _{n-2}\frac{k_{n-1}}{k_{n-2}},
\\
H_{i}^{\ast } &=&\left( k_{n-i}H_{i-2}^{\ast }-\nabla _{V_{1}}H_{i-1}^{\ast
}\right) \frac{\varepsilon _{n-\left( i+2\right) }\varepsilon _{n-\left(
i+1\right) }}{k_{n-\left( i+1\right) }},\text{ }2\leq i\leq n-2\text{.}
\end{eqnarray*}%
Note that $\nabla _{V_{1}}H_{i-1}^{\ast }=V_{1}\left( H_{i-1}^{\ast }\right)
=H_{i-1}^{\ast \prime }$.
\end{definition}

\begin{definition}
Let $\left( M,g\right) $ be a $n$-dimensional pseudo-Riemannian manifold .
Let $f\in \digamma \left( M\right) $ and $\nabla f$ be its gradient, i.e. 
\begin{equation*}
g\left( \nabla f,X\right) =df(X)=X\left( f\right)
\end{equation*}%
for all $X\in TM$, where $\digamma \left( M\right) $ is the set of all
smooth real-valued functions on $M$. Then, we say that $f$ is eikonal
function if $f$ satisfies the eikonal equation%
\begin{equation*}
g\left( \nabla f,\nabla f\right) =\text{constant.}
\end{equation*}
\end{definition}

\begin{lemma}
The Hessian $H^{f}$ of $f\in \digamma \left( M\right) $ is the symmetric
(0,2) tensor field such that%
\begin{equation*}
H^{f}\left( X,Y\right) =g\left( \nabla _{X}\left( \func{grad}f\right)
,Y\right) \text{,}
\end{equation*}%
where $\left( M,g\right) $ is a pseudo-Riemannian manifold and $\nabla $ is
Levi-Civita connection of $M$ \cite{neill} .

The above Lemma has the following corollary.
\end{lemma}

\begin{corollary}
$H^{f}=0$ iff $\nabla f=\func{grad}f$ is parallel in $M$.
\end{corollary}

\begin{proof}
We assume that $H^{f}=0$. Since $g$ is non-degenerate metric, $\nabla
_{X}\left( \func{grad}f\right) =0$ for all $X,Y\in TM$. In other words, $%
\nabla f$ is parallel in $M$.

Conversely, if $\nabla f$ is parallel in $M$, then $\nabla _{X}\left( \func{%
grad}f\right) =0$ for all $X\in TM$. Hence, $H^{f}=0$. This completes the
proof.
\end{proof}

\begin{definition}
Let $\left( M,g\right) $ be a $n$-dimensional pseudo-Riemannian manifold and
let $\alpha \left( s\right) $ be a proper curve of order $n$ (non-null) in $%
M $. Let $f\in \digamma \left( M\right) $ be a eikonal function along curve $%
\alpha $, i.e. $g\left( \nabla f,\nabla f\right) $ is constant along curve $%
\alpha $. If the function%
\begin{equation*}
g\left( \nabla f,V_{n}\right)
\end{equation*}%
is non-zero constant function along $\alpha $, then $\alpha $ is called a $f$%
-eikonal $V_{n}$-slant helix curve, where $V_{n}$ is $n$-th Frenet Frame
field. And, $\nabla f$ is called the axis of the $f$-eikonal $V_{n}$-slant
helix curve $\alpha $.
\end{definition}

\begin{theorem}
Let $\left( M,g\right) $ be a $n$-dimensional pseudo-Riemannian manifold and
let $\alpha \left( s\right) $ be a proper curve of order $n$ (non-null) in $%
M $. Let us assume that $f\in \digamma \left( M\right) $ be a eikonal
function along curve $\alpha $, i.e. $g\left( \nabla f,\nabla f\right) =$%
constant along curve $\alpha $ and the Hessian $H^{f}$ $=0$. If $\alpha $ is
a $f$-eikonal $V_{n}$-slant helix curve with the axis $\nabla f$, then the
system 
\begin{equation}
g\left( V_{n-(i+1)},\nabla f\right) =H_{i}^{\ast }g\left( V_{n},\nabla
f\right) ,\text{ }i=1,...,n-2
\end{equation}%
holds, where $\left \{ V_{1},...,V_{n}\right \} $ and $\left \{ H_{1}^{\ast
},...,H_{n-2}^{\ast }\right \} $ are the Frenet frame and the harmonic
curvature functions of $\alpha $, respectively.
\end{theorem}

\begin{proof}
Since $\left \{ V_{1},...,V_{n}\right \} $ is the orthonormal frame of the
curve $\alpha $ in $M$, $\nabla f$ can be expressed in the form%
\begin{equation}
\nabla f=\lambda _{1}V_{1}+...+\lambda _{n}V_{n}\text{.}
\end{equation}%
By using the definition of $\ f$-eikonal $V_{n}$-slant helix curve and
(3.2), we get%
\begin{equation}
g\left( V_{n},\nabla f\right) =\lambda _{n}\varepsilon _{n-1}=\text{constant.%
}
\end{equation}%
If we take the derivative in each part of (3.3) in the direction $V_{1}$ in $%
M$, then we have%
\begin{equation}
g\left( \nabla _{V_{1}}\nabla f,V_{n}\right) +g\left( \nabla f,\nabla
_{V_{1}}V_{n}\right) =0\text{.}
\end{equation}%
On the other hand, from Corollary 3.1, $\nabla f$ is parallel in $M$. That
is,$\nabla _{X}\nabla f=0$ for all $X\in TM$. So, $\nabla _{V_{1}}\nabla f=0$
for $X=V_{1}$. Hence, by using (3.4) and Frenet formulas, we obtain%
\begin{equation}
-\varepsilon _{n-2}\varepsilon _{n-1}k_{n-1}g\left( \nabla f,V_{n-1}\right)
=0\text{.}
\end{equation}%
And, since $\varepsilon _{n-2}\varepsilon _{n-1}k_{n-1}$ is different from
zero, from (3.5), we get%
\begin{equation}
g\left( \nabla f,V_{n-1}\right) =0\text{.}
\end{equation}%
By taking the derivative in each part of (3.6) in the direction $V_{1}$ in $%
M $, we can write the equality%
\begin{equation}
g\left( \nabla _{V_{1}}\nabla f,V_{n-1}\right) +g\left( \nabla f,\nabla
_{V_{1}}V_{n-1}\right) =0\text{.}
\end{equation}%
And, since $\nabla _{V_{1}}\nabla f=0$, by using (3.7) and Frenet formulas,
we obtain%
\begin{equation}
-\varepsilon _{n-3}\varepsilon _{n-2}k_{n-2}g\left( \nabla f,V_{n-2}\right)
+k_{n-1}g\left( \nabla f,V_{n}\right) =0\text{.}
\end{equation}%
Therefore, from (3.8), we have%
\begin{eqnarray}
g\left( \nabla f,V_{n-2}\right) &=&\frac{k_{n-1}}{\varepsilon
_{n-3}\varepsilon _{n-2}k_{n-2}}g\left( \nabla f,V_{n}\right)  \notag \\
g\left( \nabla f,V_{n-2}\right) &=&\frac{\varepsilon _{n-3}\varepsilon _{n-2}%
}{\left( \varepsilon _{n-3}\right) ^{2}\left( \varepsilon _{n-2}\right) ^{2}}%
\text{.}\frac{k_{n-1}}{k_{n-2}}g\left( \nabla f,V_{n}\right)  \notag \\
g\left( \nabla f,V_{n-2}\right) &=&\varepsilon _{n-3}\varepsilon _{n-2}\frac{%
k_{n-1}}{k_{n-2}}g\left( \nabla f,V_{n}\right) \text{.}
\end{eqnarray}%
Moreover, since $H_{1}^{\ast }=\varepsilon _{n-3}\varepsilon _{n-2}\dfrac{%
k_{n-1}}{k_{n-2}}$, from (3.9), we can write%
\begin{equation*}
g\left( \nabla f,V_{n-2}\right) =H_{1}^{\ast }g\left( \nabla f,V_{n}\right) 
\text{.}
\end{equation*}%
It follows that the equality (3.1) is true for $i=1$. According to the
induction theory, let us assume that the equality (3.1) is true for all $k$,
where $1\leq k\leq i$ for some positive integers $i$. Then, we will prove
that the equality (3.1) is true for $i+1$. Since the equality (3.1) is true
for some positive integers $i$, we can write%
\begin{equation}
g\left( V_{n-\left( i+1\right) },\nabla f\right) =H_{i}^{\ast }g\left(
\nabla f,V_{n}\right)
\end{equation}%
for some positive integers $i$. If we take derivative in each part of (3.10)
in the direction $V_{1}$ in $M$, we have%
\begin{equation}
g\left( \nabla _{V_{1}}V_{n-(i+1)},\nabla f\right) +g\left(
V_{n-(i+1)},\nabla _{V_{1}}\nabla f\right) =V_{1}\left( H_{i}^{\ast }g\left(
\nabla f,V_{n}\right) \right) \text{.}
\end{equation}%
And, by using (3.11) and Frenet formulas, we get the equality%
\begin{eqnarray}
V_{1}\left( H_{i}^{\ast }g\left( \nabla f,V_{n}\right) \right)
&=&-\varepsilon _{n-\left( i+3\right) }\varepsilon _{n-\left( i+2\right)
}k_{n-(i+2)}g\left( V_{n-(i+2)},\nabla f\right) \\
&&+k_{n-(i+1)}g\left( V_{n-i},\nabla f\right) +g\left( V_{n-(i+1)},\nabla
_{V_{1}}\nabla f\right) \text{.}  \notag
\end{eqnarray}%
Morever, $\nabla _{V_{1}}\nabla f=0$. Hence, from (3.12), we can write%
\begin{equation}
-\varepsilon _{n-\left( i+3\right) }\varepsilon _{n-\left( i+2\right)
}k_{n-(i+2)}g\left( V_{n-(i+2)},\nabla f\right) +k_{n-(i+1)}g\left(
V_{n-i},\nabla f\right) =V_{1}\left( H_{i}^{\ast }g\left( \nabla
f,V_{n}\right) \right) \text{.}
\end{equation}%
And, from (3.13), we obtain%
\begin{eqnarray}
g\left( V_{n-(i+2)},\nabla f\right) &=&\left \{ -V_{1}\left( H_{i}^{\ast
}g\left( \nabla f,V_{n}\right) \right) +k_{n-(i+1)}g\left( V_{n-i},\nabla
f\right) \right \} . \\
&&\frac{\varepsilon _{n-\left( i+3\right) }\varepsilon _{n-\left( i+2\right)
}}{\left( \varepsilon _{n-\left( i+3\right) }\right) ^{2}\left( \varepsilon
_{n-\left( i+2\right) }\right) ^{2}}\frac{1}{k_{n-(i+2)}}  \notag \\
&=&\left \{ -V_{1}\left( H_{i}^{\ast }g\left( \nabla f,V_{n}\right) \right)
+k_{n-(i+1)}g\left( V_{n-i},\nabla f\right) \right \} \varepsilon _{n-\left(
i+3\right) }\varepsilon _{n-\left( i+2\right) }\frac{1}{k_{n-(i+2)}}\text{.}
\notag
\end{eqnarray}%
On the other hand, since the equality (3.1) is true for $i-1$ according to
the induction hypothesis, we have%
\begin{equation}
g\left( V_{n-i},\nabla f\right) =H_{i-1}^{\ast }g\left( \nabla
f,V_{n}\right) \text{.}
\end{equation}%
Therefore, by using (3.3), (3.14) and (3.15), we get 
\begin{equation}
g\left( V_{n-(i+2)},\nabla f\right) =\left \{ -V_{1}\left( H_{i}^{\ast
}\right) +k_{n-(i+1)}H_{i-1}^{\ast }\right \} \varepsilon _{n-\left(
i+3\right) }\varepsilon _{n-\left( i+2\right) }\frac{1}{k_{n-(i+2)}}g\left(
\nabla f,V_{n}\right)
\end{equation}%
Moreover, we obtain%
\begin{equation}
H_{i+1}^{\ast }=\left \{ k_{n-(i+1)}H_{i-1}^{\ast }-V_{1}\left( H_{i}^{\ast
}\right) \right \} \varepsilon _{n-\left( i+3\right) }\varepsilon _{n-\left(
i+2\right) }\frac{1}{k_{n-(i+2)}}
\end{equation}%
for $i+1$ in the Definition 3.1. So, we have%
\begin{equation*}
g\left( V_{n-(i+2)},\nabla f\right) =H_{i+1}^{\ast }g\left( \nabla
f,V_{n}\right)
\end{equation*}%
by using (3.16) and (3.17). It follows that the equality (3.1) is true for $%
i+1$. Consequently, we get%
\begin{equation*}
g\left( V_{n-(i+1)},\nabla f\right) =H_{i}^{\ast }g\left( \nabla
f,V_{n}\right)
\end{equation*}%
for all $i$ according to induction theory. This completes the proof.
\end{proof}

\begin{theorem}
Let $\left( M,g\right) $ be a $n$-dimensional pseudo-Riemannian manifold and
let $\alpha \left( s\right) $ be a proper curve of order $n$ (non-null) in $%
M $. Let us assume that $f\in \digamma \left( M\right) $ be a eikonal
function along curve $\alpha $, i.e. $g\left( \nabla f,\nabla f\right) =$%
constant along curve $\alpha $ and the Hessian $H^{f}$ $=0$. If $\alpha $ is
a $f$-eikonal $V_{n}$-slant helix curve with the axis $\nabla f$, then the
axis of the curve $\alpha $%
\begin{equation*}
\nabla f=\left \{ \varepsilon _{0}H_{n-2}^{\ast }V_{1}+...+\varepsilon
_{n-3}H_{1}^{\ast }V_{n-2}+\varepsilon _{n-1}V_{n}\right \} g\left( \nabla
f,V_{n}\right) \text{,}
\end{equation*}%
where $\left \{ V_{1},V_{2},...,V_{n}\right \} $ and $\left \{ H_{1}^{\ast
},...,H_{n-2}^{\ast }\right \} $ are the Frenet frame and the harmonic
curvatures of $\alpha $, respectively.
\end{theorem}

\begin{proof}
Since $\alpha $ is a $f$-eikonal $V_{n}$-slant helix curve , we can write%
\begin{equation}
g\left( \nabla f,V_{n}\right) =\text{constant.}
\end{equation}%
If we take the derivative in each part of (3.18) in the direction $V_{1}$ in 
$M$, then we have%
\begin{equation}
g\left( \nabla _{V_{1}}\nabla f,V_{n}\right) +g\left( \nabla f,\nabla
_{V_{1}}V_{n}\right) =0\text{.}
\end{equation}%
On the other hand, from Corollary 3.1, $\nabla f$ is parallel in $M$. That's
why, $\nabla _{V_{1}}\nabla f=0$. Then, we obtain%
\begin{equation}
-\varepsilon _{n-2}\varepsilon _{n-1}k_{n-1}g\left( \nabla f,V_{n-1}\right)
=0
\end{equation}%
by using (3.19) and Frenet formulas. Since $\varepsilon _{n-2}\varepsilon
_{n-1}k_{n-1}$ is positive function, (3.20) implies that%
\begin{equation*}
g\left( \nabla f,V_{n-1}\right) =0\text{.}
\end{equation*}%
Hence, we can write the axis of $\alpha $ as%
\begin{equation}
\nabla f=\lambda _{1}V_{1}+\lambda _{2}V_{2}+...+\lambda
_{n-2}V_{n-2}+\lambda _{n}V_{n}\text{.}
\end{equation}%
Moreover, from (3.21), we get%
\begin{eqnarray*}
\varepsilon _{0}\lambda _{1} &=&g\left( \nabla f,V_{1}\right) \\
\varepsilon _{1}\lambda _{2} &=&g\left( \nabla f,V_{2}\right) \\
&&. \\
&&. \\
&&. \\
\varepsilon _{n-3}\lambda _{n-2} &=&g\left( \nabla f,V_{n-2}\right) \\
\varepsilon _{n-1}\lambda _{n} &=&g\left( \nabla f,V_{n}\right)
\end{eqnarray*}%
by using the metric $g$. On the other hand, from Theorem 3.1, we know that 
\begin{eqnarray}
\lambda _{1} &=&g\left( \nabla f,V_{1}\right) =\varepsilon _{0}H_{n-2}^{\ast
}g\left( \nabla f,V_{n}\right) \\
\lambda _{2} &=&g\left( \nabla f,V_{2}\right) =\varepsilon _{1}H_{n-3}^{\ast
}g\left( \nabla f,V_{n}\right)  \notag \\
&&.  \notag \\
&&.  \notag \\
&&.  \notag \\
\lambda _{n-2} &=&g\left( \nabla f,V_{n-2}\right) =\varepsilon
_{n-3}H_{1}^{\ast }g\left( \nabla f,V_{n}\right)  \notag \\
\lambda _{n} &=&\frac{1}{\varepsilon _{n-1}}g\left( \nabla f,V_{n}\right) =%
\frac{\varepsilon _{n-1}}{\left( \varepsilon _{n-1}\right) ^{2}}g\left(
\nabla f,V_{n}\right)  \notag \\
&=&\varepsilon _{n-1}g\left( \nabla f,V_{n}\right) \text{.}  \notag
\end{eqnarray}%
Thus, it can be easily obtained the axis of the curve $\alpha $ as 
\begin{equation*}
\nabla f=\left \{ \varepsilon _{0}H_{n-2}^{\ast }V_{1}+...+\varepsilon
_{n-3}H_{1}^{\ast }V_{n-2}+\varepsilon _{n-1}V_{n}\right \} g\left( \nabla
f,V_{n}\right) \text{,}
\end{equation*}%
by making use of the equality (3.21) and the system (3.22). This completes
the proof.
\end{proof}

\begin{theorem}
Let $\left( M,g\right) $ be a $n$-dimensional pseudo-Riemannian manifold and
let $\alpha \left( s\right) $ be a proper curve of order $n$ (non-null) in $%
M $. Let us assume that $f\in \digamma \left( M\right) $ be a eikonal
function along curve $\alpha $, i.e. $g\left( \nabla f,\nabla f\right) =$%
constant along curve $\alpha $ and the Hessian $H^{f}$ $=0$. If $\alpha $ is
a $f$-eikonal $V_{n}$-slant helix curve, then $H_{n-2}^{\ast }\neq 0$ and $%
\varepsilon _{n-3}H_{1}^{\ast 2}+\varepsilon _{n-4}H_{2}^{\ast
2}+...+\varepsilon _{0}H_{n-2}^{\ast 2}$ is non-zero constant, where $%
\left
\{ H_{1}^{\ast },...,H_{n-2}^{\ast }\right \} $ is the harmonic
curvatures of $\alpha $.
\end{theorem}

\begin{proof}
Let $\alpha $ be a $f$-eikonal $V_{n}$-slant helix curve and $\left \{
V_{1},...,V_{n}\right \} $ be the Frenet frame of $\alpha $. Then, from
Theorem 3.2, we know that 
\begin{equation}
\nabla f=\left \{ \varepsilon _{0}H_{n-2}^{\ast }V_{1}+...+\varepsilon
_{n-3}H_{1}^{\ast }V_{n-2}+\varepsilon _{n-1}V_{n}\right \} g\left( \nabla
f,V_{n}\right) \text{.}
\end{equation}%
Therefore, from (3.23), we can write%
\begin{equation}
g\left( \nabla f,\nabla f\right) =\left( g\left( \nabla f,V_{n}\right)
\right) ^{2}\left( \varepsilon _{0}^{3}H_{n-2}^{\ast 2}+...+\varepsilon
_{n-3}^{3}H_{1}^{\ast 2}+\varepsilon _{n-1}^{3}\right) \text{.}
\end{equation}%
Moreover, by the definition of metric tensor, we have 
\begin{equation*}
\left \vert g\left( \nabla f,\nabla f\right) \right \vert =\left \Vert
\nabla f\right \Vert ^{2}\text{.}
\end{equation*}%
According to this Theorem, $\alpha $ is a $f$-eikonal $V_{n}$-slant helix
curve. So, $\left \Vert \nabla f\right \Vert =$constant and $g\left( \nabla
f,V_{n}\right) $ $=$ non-zero constant along $\alpha $. Hence, from (3.24),
we obtain%
\begin{equation*}
\varepsilon _{0}^{3}H_{n-2}^{\ast 2}+...+\varepsilon _{n-3}^{3}H_{1}^{\ast
2}+\varepsilon _{n-1}^{3}=\text{constant.}
\end{equation*}%
In other words,%
\begin{equation*}
\varepsilon _{0}H_{n-2}^{\ast 2}+...+\varepsilon _{n-3}H_{1}^{\ast 2}=\text{%
constant.}
\end{equation*}%
Now, we will show that $H_{n-2}^{\ast }\neq 0$ . We assume that $%
H_{n-2}^{\ast }=0$. Then, for $i=n-2$ in (3.1),%
\begin{equation}
g\left( V_{1},\nabla f\right) =H_{n-2}^{\ast }g\left( \nabla f,V_{n}\right)
=0\text{.}
\end{equation}%
If we take derivative in each part of (3.25) in the direction $V_{1}$ on $M$%
, then we have%
\begin{equation}
g\left( \nabla _{V_{1}}V_{1},\nabla f\right) +g\left( V_{1},\nabla
_{V_{1}}\nabla f\right) =0\text{.}
\end{equation}%
On the other hand, from Corollary 3.1, $\nabla f$ is parallel in $M$. That's
why $\nabla _{V_{1}}\nabla f=0$. Then, from (3.26), we have $g\left( \nabla
_{V_{1}}V_{1},\nabla f\right) =k_{1}g\left( V_{2},\nabla f\right) =0$ by
using the Frenet formulas. Since $k_{1}$ is positive, $g\left( V_{2},\nabla
f\right) =0$. Now, for $i=n-3$ in (3.1),%
\begin{equation*}
g\left( V_{2},\nabla f\right) =H_{n-3}^{\ast }g\left( V_{n},\nabla f\right) 
\text{.}
\end{equation*}%
And, since $g\left( V_{2},\nabla f\right) $ $=0$, $H_{n-3}^{\ast }=0$.
Continuing this process, we get $H_{1}^{\ast }=0$. Let us recall that $%
H_{1}^{\ast }=\varepsilon _{n-3}\varepsilon _{n-2}\frac{k_{n-1}}{k_{n-2}}$,
thus we have a contradiction because all the curvatures are nowhere zero.
Consequently, $H_{n-2}^{\ast }\neq 0$. This completes the proof.
\end{proof}

\begin{lemma}
Let $\left( M,g\right) $ be a $n$-dimensional pseudo-Riemannian manifold and
let $\alpha \left( s\right) $ be a proper curve of order $n$ (non-null) in $%
M $. Let us assume that $H_{n-2}^{\ast }\neq 0$ for $i=n-2$. Then, $%
\varepsilon _{n-3}H_{1}^{\ast 2}+\varepsilon _{n-4}H_{2}^{\ast
2}+...+\varepsilon _{0}H_{n-2}^{\ast 2}$ is non-zero constant if and only if 
$V_{1}\left( H_{n-2}^{\ast }\right) =H_{n-2}^{\ast \prime
}=k_{1}H_{n-3}^{\ast }$, where $V_{1}$ and $\left \{ H_{1}^{\ast
},...,H_{n-2}^{\ast }\right \} $ are the unit tangent vector field and the
harmonic curvatures of $\alpha $, respectively.
\end{lemma}

\begin{proof}
First,we assume that $\varepsilon _{n-3}H_{1}^{\ast 2}+\varepsilon
_{n-4}H_{2}^{\ast 2}+...+\varepsilon _{0}H_{n-2}^{\ast 2}$ is non-zero
constant. Consider the following functions given in Definition 3.1%
\begin{equation*}
H_{i}^{\ast }=\left( k_{n-i}H_{i-2}^{\ast }-H_{i-1}^{\ast \prime }\right) 
\frac{\varepsilon _{n-\left( i+2\right) }\varepsilon _{n-\left( i+1\right) }%
}{k_{n-\left( i+1\right) }}
\end{equation*}%
for $3\leq i\leq n-2$. So, from the equality, we can write%
\begin{equation}
k_{n-\left( i+1\right) }H_{i}^{\ast }=\varepsilon _{n-\left( i+2\right)
}\varepsilon _{n-\left( i+1\right) }\left( k_{n-i}H_{i-2}^{\ast
}-H_{i-1}^{\ast \prime }\right) \text{.}
\end{equation}%
Hence, in (3.27), if we take $i+1$ instead of $i$, we get%
\begin{equation}
\varepsilon _{n-\left( i+3\right) }\varepsilon _{n-\left( i+2\right)
}H_{i}^{\ast \prime }=\varepsilon _{n-\left( i+3\right) }\varepsilon
_{n-\left( i+2\right) }k_{n-\left( i+1\right) }H_{i-1}^{\ast }-k_{n-\left(
i+2\right) }H_{i+1}^{\ast },\text{ }2\leq i\leq n-3
\end{equation}%
together with%
\begin{equation*}
H_{1}^{\ast \prime }=-\frac{1}{\varepsilon _{n-4}\varepsilon _{n-3}}%
k_{n-3}H_{2}^{\ast }
\end{equation*}%
or 
\begin{equation}
H_{1}^{\ast \prime }=-\varepsilon _{n-4}\varepsilon _{n-3}k_{n-3}H_{2}^{\ast
}\text{.}
\end{equation}%
On the other hand, since $\varepsilon _{n-3}H_{1}^{\ast 2}+\varepsilon
_{n-4}H_{2}^{\ast 2}+...+\varepsilon _{0}H_{n-2}^{\ast 2}$ is constant, we
have%
\begin{equation*}
\varepsilon _{n-3}H_{1}^{\ast }H_{1}^{\ast \prime }+\varepsilon
_{n-4}H_{2}^{\ast }H_{2}^{\ast \prime }+...+\varepsilon _{0}H_{n-2}^{\ast
}H_{n-2}^{\ast \prime }=0
\end{equation*}%
and so, 
\begin{equation}
\varepsilon _{0}H_{n-2}^{\ast }H_{n-2}^{\ast \prime }=-\varepsilon
_{n-3}H_{1}^{\ast }H_{1}^{\ast \prime }-\varepsilon _{n-4}H_{2}^{\ast
}H_{2}^{\ast \prime }-...-\varepsilon _{1}H_{n-3}^{\ast }H_{n-3}^{\ast
\prime }\text{.}
\end{equation}%
By using (3.28) and (3.29), we obtain 
\begin{equation}
H_{1}^{\ast }H_{1}^{\ast \prime }=-\varepsilon _{n-4}\varepsilon
_{n-3}k_{n-3}H_{1}^{\ast }H_{2}^{\ast }
\end{equation}%
and%
\begin{equation}
\varepsilon _{n-\left( i+3\right) }\varepsilon _{n-\left( i+2\right)
}H_{i}^{\ast }H_{i}^{\ast \prime }=\varepsilon _{n-\left( i+3\right)
}\varepsilon _{n-\left( i+2\right) }k_{n-\left( i+1\right) }H_{i-1}^{\ast
}H_{i}^{\ast }-k_{n-\left( i+2\right) }H_{i}^{\ast }H_{i+1}^{\ast },\text{ }%
2\leq i\leq n-3\text{.}
\end{equation}%
Therefore, by using (3.30), (3.31) and (3.32), an algebraic calculus shows
that 
\begin{equation*}
\varepsilon _{0}H_{n-2}^{\ast }H_{n-2}^{\ast \prime }=\varepsilon
_{0}k_{1}H_{n-3}^{\ast }H_{n-2}^{\ast }
\end{equation*}%
or%
\begin{equation*}
H_{n-2}^{\ast }H_{n-2}^{\ast \prime }=k_{1}H_{n-3}^{\ast }H_{n-2}^{\ast }%
\text{.}
\end{equation*}%
Since $H_{n-2}^{\ast }\neq 0$, we get the relation%
\begin{equation*}
H_{n-2}^{\ast \prime }=k_{1}H_{n-3}^{\ast }\text{.}
\end{equation*}

Conversely, we assume that%
\begin{equation}
H_{n-2}^{\ast \prime }=k_{1}H_{n-3}^{\ast }\text{.}
\end{equation}%
By using (3.33) and $H_{n-2}^{\ast }\neq 0$, we can write%
\begin{equation}
H_{n-2}^{\ast }H_{n-2}^{\ast \prime }=k_{1}H_{n-2}^{\ast }H_{n-3}^{\ast }
\end{equation}%
From (3.32), we have the following equation sysytem:%
\begin{eqnarray*}
\text{for }i &=&n-3\text{, \  \  \  \  \ }\varepsilon _{1}H_{n-3}^{\ast
}H_{n-3}^{\ast \prime }=\varepsilon _{1}k_{2}H_{n-4}^{\ast }H_{n-3}^{\ast
}-\varepsilon _{0}k_{1}H_{n-3}^{\ast }H_{n-2}^{\ast }\text{,} \\
\text{for }i &=&n-4\text{, \  \  \  \  \ }\varepsilon _{2}H_{n-4}^{\ast
}H_{n-4}^{\ast \prime }=\varepsilon _{2}k_{3}H_{n-5}^{\ast }H_{n-4}^{\ast
}-\varepsilon _{1}k_{2}H_{n-4}^{\ast }H_{n-3}^{\ast }\text{,} \\
\text{for }i &=&n-5\text{, \  \  \  \  \ }\varepsilon _{3}H_{n-5}^{\ast
}H_{n-5}^{\ast \prime }=\varepsilon _{3}k_{4}H_{n-6}^{\ast }H_{n-5}^{\ast
}-\varepsilon _{2}k_{3}H_{n-5}^{\ast }H_{n-4}^{\ast }\text{,} \\
&&\cdot \\
&&\cdot \\
&&\cdot \\
\text{for }i &=&2\text{, \  \  \  \  \  \  \  \  \ }\varepsilon _{n-4}H_{2}^{\ast
}H_{2}^{\ast \prime }=\varepsilon _{n-4}k_{n-3}H_{1}^{\ast }H_{2}^{\ast
}-\varepsilon _{n-5}k_{n-4}H_{2}^{\ast }H_{3}^{\ast }\text{ .}
\end{eqnarray*}%
Moreover, from (3.31) and (3.34), we obtain%
\begin{equation}
\varepsilon _{n-3}H_{1}^{\ast }H_{1}^{\ast \prime }=-\varepsilon
_{n-4}k_{n-3}H_{1}^{\ast }H_{2}^{\ast }
\end{equation}%
and%
\begin{equation}
\varepsilon _{0}H_{n-2}^{\ast }H_{n-2}^{\ast \prime }=\varepsilon
_{0}k_{1}H_{n-2}^{\ast }H_{n-3}^{\ast }\text{.}
\end{equation}%
So, by using the above equation system, (3.35) and (3.36), an algebraic
calculus shows that%
\begin{equation}
\varepsilon _{n-3}H_{1}^{\ast }H_{1}^{\ast \prime }+\varepsilon
_{n-4}H_{2}^{\ast }H_{2}^{\ast \prime }+...+\varepsilon _{0}H_{n-2}^{\ast
}H_{n-2}^{\ast \prime }=0\text{.}
\end{equation}%
And, by integrating (3.37), we can easily say that%
\begin{equation*}
\varepsilon _{n-3}H_{1}^{\ast 2}+\varepsilon _{n-4}H_{2}^{\ast
2}+...+\varepsilon _{0}H_{n-2}^{\ast 2}
\end{equation*}%
is a non-zero constant. This completes the proof.
\end{proof}

\begin{corollary}
Let $\left( M,g\right) $ be a $n$-dimensional pseudo-Riemannian manifold and
let $\alpha \left( s\right) $ be a proper curve of order $n$ (non-null) in $%
M $. Let us assume that $f\in \digamma \left( M\right) $ be a eikonal
function along curve $\alpha $, i.e. $g\left( \nabla f,\nabla f\right) =$%
constant along curve $\alpha $ and the Hessian $H^{f}$ $=0$. If $\alpha $ is
a $f$-eikonal $V_{n}$-slant helix curve, $V_{1}\left( H_{n-2}^{\ast }\right)
=H_{n-2}^{\ast \prime }=k_{1}H_{n-3}^{\ast }$.
\end{corollary}

\begin{proof}
It is obvious by using Theorem 3.3 and Lemma 3.2.
\end{proof}

\begin{center}
\textbf{4. Conclusions}
\end{center}

In this work, it is defined $f$-eikonal $V_{n}$-slant helix by the gradient
vector field $\nabla f$ and $\nabla f$ is called as the axis of the eikonal
slant helix. Besides, it is given new characterizations on eikonal slant
helices by using the harmonic curvature functions in $n$-dimensional
pseudo-Riemannian manifolds.

On the other hand, we want to emphasize an important point. The axis $\nabla
f$ defined in this work is non-constant. If the axis $\nabla f$ \ is
considered as a constant vector field, then the eikonal slant helix defined
in this paper coincides with $V_{n}$-slant helix which is introduced in \cite%
{gok1}. Also, if $\nabla f$ is a Levi-Civita parallel vector field, then
eikonal slant helix is a LC-slant helix defined by in \cite{oz, ali2}

\end{document}